\documentclass[12pt]{amsart}
\usepackage[utf8]{inputenc}
\usepackage[shortlabels]{enumitem}
\usepackage{amssymb}
\usepackage{mathrsfs}
\usepackage{stmaryrd}
\usepackage[all]{xy}
\usepackage[margin=1in]{geometry} 
\usepackage[bookmarks, bookmarksdepth=2, colorlinks=true, linkcolor=blue, citecolor=blue, urlcolor=blue]{hyperref}
\usepackage{eucal}
\usepackage{contour}
\usepackage{shuffle}
\usepackage[normalem]{ulem}
\usepackage{tikz}
\usetikzlibrary{positioning}
\usetikzlibrary{arrows.meta}
\usetikzlibrary{decorations.markings}
\usetikzlibrary{plotmarks}
\usepackage{dsfont}

\numberwithin{equation}{section}
\newtheorem{theorem}[equation]{Theorem}

\newtheorem{proposition}[equation]{Proposition}
\newtheorem{lemma}[equation]{Lemma}
\newtheorem{corollary}[equation]{Corollary}

\theoremstyle{definition}
\newtheorem{rmk}[equation]{Remark}
\newenvironment{remark}[1][]{\begin{rmk}[#1] \pushQED{\qed}}{\popQED \end{rmk}}
\newtheorem{eg}[equation]{Example}

\newtheorem{defnaux}[equation]{Definition}
\newenvironment{definition}[1][]{\begin{defnaux}[#1]\pushQED{\qed}}{\popQED \end{defnaux}}

\newcommand{\cA}{\mathcal{A}}
\newcommand{\fA}{\mathfrak{A}}

\newcommand{\BB}{\mathbb{B}}

\newcommand{\fB}{\mathfrak{B}}

\newcommand{\cC}{\mathcal{C}}

\newcommand{\cE}{\mathcal{E}}

\newcommand{\sE}{\mathscr{E}}
\newcommand{\bF}{\mathbf{F}}
\newcommand{\cF}{\mathcal{F}}

\newcommand{\rM}{\mathrm{M}}

\newcommand{\bQ}{\mathbf{Q}}

\newcommand{\cR}{\mathcal{R}}

\newcommand{\bS}{\mathbf{S}}

\newcommand{\fS}{\mathfrak{S}}

\newcommand{\cT}{\mathcal{T}}

\newcommand{\cV}{\mathcal{V}}

\newcommand{\cW}{\mathcal{W}}

\newcommand{\fX}{\mathfrak{X}}

\newcommand{\bZ}{\mathbf{Z}}

\newcommand{\arxiv}[1]{\href{http://arxiv.org/abs/#1}{{\tiny\tt arXiv:#1}}}
\newcommand{\DOI}[1]{\href{http://doi.org/#1}{\color{purple}{\tiny\tt DOI:#1}}}

\newcommand{\defn}[1]{\emph{#1}}

\let\ul\underline
\renewcommand{\phi}{\varphi}
\renewcommand{\emptyset}{\varnothing}
\DeclareMathOperator{\im}{im} 
\DeclareMathOperator{\End}{End}

\DeclareMathOperator{\Mat}{Mat}

\DeclareMathOperator{\Hom}{Hom}

\newcommand{\op}{\mathrm{op}}
\renewcommand{\Vec}{\mathrm{Vec}}
\newcommand{\GL}{\mathbf{GL}}

\newcommand{\bbone}{\mathds{1}}

\newcommand{\bone}{\mathbf{1}}
\newcommand{\uotimes}{\mathbin{\ul{\otimes}}}

\newcommand{\FI}{\mathbf{FI}}
\newcommand{\FS}{\mathbf{FS}}

\contourlength{1pt}

\contourlength{0.8pt}
\newcommand{\myuline}[1]{%
  \uline{\phantom{#1}}%
  \llap{\contour{white}{#1}}%
}
\DeclareMathOperator{\uRep}{\text{\myuline{\rm Rep}}}
\DeclareMathOperator{\uPerm}{\ul{Perm}}

\title[On the representation theory of the symmetry group of the Cantor set]{On the representation theory of \\ the symmetry group of the Cantor set}
\author{Andrew Snowden}
\thanks{The author was supported by NSF grant DMS-2301871.}
\date{April 8, 2024}

\begin{document}

\begin{abstract}
In previous work with Harman, we introduced a new class of representations for an oligomorphic group $G$, depending on an auxiliary piece of data called a measure. In this paper, we look at this theory when $G$ is the symmetry group of the Cantor set. We show that $G$ admits exactly two measures $\mu$ and $\nu$. The representation theory of $(G, \mu)$ is the linearization of the category of $\bF_2$-vector spaces, studied in recent work of the author and closely connected to work of Kuhn and Kov\'acs. The representation theory of $(G, \nu)$ is the linearization of the category of vector spaces over the Boolean semi-ring (or, equivalently, the correspondence category), studied by Bouc--Th\'evenaz. The latter case yields an important counterexample in the general theory.
\end{abstract}

\maketitle
\tableofcontents

\section{Introduction}

\subsection{Background}

Let $G$ be an oligomorphic permutation group, i.e., $G$ acts by permutations on a set $\Omega$ and has finitely many orbits on $\Omega^n$ for all $n$. In recent work with Harman \cite{repst}, we introduced a class of generalized ``permutation modules'' for $G$ depending on an auxiliary piece of data $\mu$ called a measure. These modules form a rigid tensor category $\uPerm_k(G; \mu)$. In certain cases, this category admits an abelian envelope.

For example, if $G$ is the infinite symmetric group then there is a 1-parameter family of measures $\mu_t$ for $G$, and the category $\uPerm_k(G; \mu)$ is (essentially) Deligne's interpolation category $\uRep(\fS_t)$ introduced in \cite{Deligne3}. Other choices for $G$ recover the interpolation categories later studied by Knop \cite{Knop,Knop2}. Still other choices for $G$ lead to entirely new categories. For example, the group of order-preserving bijections of $\bQ$ gives rise to the Delannoy category, studied in depth in \cite{line}; see \cite{circle} for a related example.

In this paper, we apply the theory of \cite{repst} when $G$ is the symmetry group of the Cantor set. This group has a somewhat different nature than groups previously considered in relation to \cite{repst}, and (perhaps because of this) we observe some new phenomena.

\subsection{Results} \label{ss:results}

Let $G$ be the group of self-homeomorphisms of the Cantor set $\fX$. The group $G$ is oligomorphic via its action on the set $\Omega$ of proper non-empty clopen subsets of $\fX$. We summarize our main results.

\textit{(a) Measures.} We show that $G$ admits exactly two measures $\mu$ and $\nu$ (both valued in $\bQ$). To be a bit more precise, we show that the ring $\Theta(G)$ introduced in \cite[\S 4]{repst}, which carries the universal measure for $G$, is isomorphic to $\bZ_{(2)} \times \bQ$.

\textit{(b) Representation theory at $\mu$.} Let $k$ be a field of characteristic~0. We show that the Karoubi envelope of $\uPerm_k(G; \mu)$ is semi-simple and pre-Tannakian. In fact, we show that this category is equivalent to (a slight modification of) the $k$-linearization of the category of finite dimensional $\bF_2$-vector spaces. This category was studied in our recent paper \cite{dblexp}, where it was shown to be semi-simple and pre-Tannakian; this result ultimately rests on work of Kuhn \cite{Kuhn} and Kov\'acs \cite{Kovacs} on the structure of the monoid algebra $k[\rM_n(\bF_2)]$, where $\rM_n(\bF_2)$ is the monoid of $n \times n$ matrices over $\bF_2$. This category is of interest as it was the first example of a pre-Tannakian category of double exponential growth. The simple objects of $\uPerm_k(G;\mu)^{\rm kar}$ are naturally in bijection with simple $k[\GL_n(\bF_2)]$-modules, for $n \ge 1$.

\textit{(c) Representation theory at $\nu$.} We show that the Karoubi envelope of $\uPerm_k(G; \nu)$ is equivalent to (a slight modification of) the $k$-linearization of the category of finite dimensional ``vector spaces'' over the two element Boolean semiring $\BB$. This category has been studied in depth by Bouc and Th\'evenaz \cite{BT1, BT2, BT3, BT4}. The endomorphism rings in $\uPerm_k(G; \nu)$ are (closely related to) the monoid algebras $k[\rM_n(\BB)]$, where $\rM_n(\BB)$ is the monoid of $n \times n$ Boolean matrices. These rings are known to not be semi-simple in general; in fact, Bremner \cite{Bremner} explicitly computes the (non-zero) Jacobson radical for $n=3$. This implies that there are nilpotent endomorphisms in $\uPerm_k(G; \mu)$ with non-zero trace, and so $\uPerm_k(G; \nu)$ does not embed into any pre-Tannakian category. This is the first example of this behavior in the context of oligomorphic groups, and one that we did not expect. See the discussion in \S \ref{ss:nu-thm} for more details.

\subsection{Remarks}

We make a few additional remarks:

(a) The group $G$ has double exponential growth, in the sense that the number of $G$-orbits on $\Omega^n$ has such growth; see Remark~\ref{rmk:growth}. This is the fastest growing oligomorphic group for which measures have been analyzed\footnote{In later work \cite[\S 8]{regcat}, we analyze measures on some faster growing groups.}; all previous cases have growth bounded by $c^{n^2}$ for some constant $c$. It is easy to construct oligomorphic groups with arbitrarily fast growth (see \cite[\S 3.24]{CameronBook}), but challenging to construct measures on fast-growing groups.

(b) The most common way of constructing oligomorphic groups is with Fra\"iss\'e limits. For example, the Fra\"iss\'e limit of the class of finite sets is a countably infinite set, and its automorphism group is the infinite symmetric group; in the context of \cite{repst}, this group leads to Deligne's interpolation category $\uRep(\fS_t)$. Before this paper, all oligomorphic groups studied in relation to \cite{repst} naturally arise naturally from (ordinary) Fra\"iss\'e limits.

The Cantor set can be viewed as the \emph{projective} Fra\"iss\'e limit of the class of finite sets, in the sense of \cite{IS}. Thus, in a way, the group $G$ we study is dual to the infinite symmetric group; this point of view shows that $G$ is a quite natural group to consider in relation to \cite{repst}. Our results suggest that it could be profitable to apply the theory of \cite{repst} to other oligomorphic groups arising from projective Fra\"iss\'e limits.

(c) Let $\cE$ be a semi-simple pre-Tannakian category over a finite field $\bF$, and let $k[\cE]$ be the linearization of $\cE$ over a field $k$ of characteristic~0 (see \cite[\S 2.2]{dblexp}). In \cite{dblexp}, we show that certain closely related categories $k[\cE]^{\sharp}_{\chi}$ are semi-simple pre-Tannakian categories. If $\cE$ arises from an oligomorphic group $H$, then so do the $k[\cE]^{\sharp}_{\chi}$ categories. This paper examines the simplest possible case, namely, when $\bF=\bF_2$ and $H$ is the trivial group, so that $\cE$ is simply the category of finite dimensional vector spaces; in this case, $k[\cE]^{\sharp}_{\chi}$ arises from our group $G$ by \S \ref{ss:results}(b). The case where $\bF=\bF_2$ and $H$ is a general group is discussed in \cite[\S 8]{regcat}.

(d) In \cite{Knop2}, Knop defines a notion of ``degree function'' on a regular category $\cA$. Given a degree function $\delta$, he constructs a tensor category $\cT(\cA, \delta)$. This formalism allowed Knop to easily construct many new interpolation categories. The axioms for a degree function are quite similar to the axioms for measures in the sense of \cite{repst}, though there are important differences in the two approaches.

In \cite[\S 8, Example~3]{Knop2}, Knop discusses the case when $\cA$ is the category of non-empty finite sets. Here there is a unique degree function $\delta$. This case corresponds to our $(G, \nu)$, and Knop's category $\cT(\cA, \delta)$ is closely related to our $\uPerm_k(G; \nu)$. Knop left open the question of whether his category was semi-simple; our results on $\uPerm_k(G; \nu)$ show that it is not.

There is no degree function corresponding to the measure $\mu$. We thus observe a stark difference between Knop's approach and the one in \cite{repst} in this case. This is explained in more detail in \cite{regcat}.

(e) The following three categories are closely related:
\begin{itemize}
\item The category of $\FI$-modules, introduced in \cite{CEF}.
\item The category of smooth (or algebraic) representations of the infinite symmetric group, as studied in, e.g., \cite[\S 6]{infrank}.
\item Deligne's interpolation categories $\uRep(\fS_t)$.
\end{itemize}
There are many connections between these categories; see, for instance, \cite{BEAH}. We would like to have a similar picture relating the following categories:
\begin{itemize}
\item The category of $\FS^{\op}$-modules.
\item The category of smooth representations of the group $G$ studied in this paper.
\item The category $\uPerm_k(G; \mu)$.
\end{itemize}
Here $\FS$ is the category of finite sets and surjective maps. Such a picture could be quite useful in studying the category of $\FS^{\op}$-modules, which has proven to be quite difficult. (See \cite{SST} for some recent work in this direction.)

\subsection{Outline}

In \S \ref{s:bg} we review general material on oligomorphic groups and tensor categories. In \S \ref{s:G} we introduce the group $G$ and study the structure of $G$-sets. In \S \ref{s:meas} we construct and classify measures on $G$. Finally, in \S \ref{s:mu} and \S \ref{s:nu}, we study the representation theories of $G$ at the two measures $\mu$ and $\nu$.

\subsection*{Acknowledgments}

We thank Nate Harman, Sophie Kriz, and Ilia Nekrasov for helpful discussions.

\section{Background} \label{s:bg}

\subsection{Tensor categories}

Let $k$ be a commutative ring. A \defn{tensor category} is an additive $k$-linear category $\cC$ equipped with a symmetric monoidal structure that is $k$-bilinear. Let $X$ be an object of $\cC$ and write $\bbone$ for the unit object of $\cC$. A \defn{dual} of an object $X$ is an object $Y$ equipped with an \defn{evaluation map} $X \otimes Y \to \bbone$ and a \defn{co-evaluation map} $\bbone \to Y \otimes X$ satisfying certain identities; see \cite[\S 2.10]{EGNO}. We say that $X$ is \defn{rigid} if it has a dual, and we say that $\cC$ is \defn{rigid} if all of its objects are.

Now suppose that $k$ is a field. The tensor category $\cC$ is \defn{pre-Tannakian} if the following conditions hold: (a) $\cC$ is abelian and all objects have finite length; (b) all $\Hom$ spaces in $\cC$ are finite dimensional over $k$; (c) $\cC$ is rigid; (d) $\End(\bbone)=k$.

\subsection{Oligomorphic groups}

An \defn{oligomorphic group} is a permutation group $(G, \Omega)$ such that $G$ has finitely many orbits on $\Omega^n$ for all $n \ge 0$. Fix such a group. For a finite subset $A$ of $\Omega$, let $G(A)$ be the subgroup of $G$ fixing each element of $A$. These subgroups form a neighborhood basis of the identity on a topology for $G$. This topology has the following three properties: (a) it is Hausdorff; (b) it is non-archimedean (open subgroups form a neighborhood basis of the identity); and (c) it is Roelcke pre-compact. See \cite[\S 2.2]{repst} for details. An \defn{admissible group} is a topological group satisfying these three conditions. While most admissible groups we care about are oligomorphic, the defining permutation representation will not play a distinguished role, and so it is more natural to simply work with the topological group.

Let $G$ be an admissible group. We say that an action of $G$ on a set $X$ is \defn{smooth} if the stabilizer in $G$ of any point in $X$ is an open subgroup. We use the term ``$G$-set'' to mean ``set equipped with a smooth action of $G$.'' We say that a $G$-set is \defn{finitary} if it has finitely many orbits. We let $\bS(G)$ be the category of finitary $G$-sets. This category is closed under fiber products. See \cite[\S 2.3]{repst} for details.

\subsection{Measures} \label{ss:meas}

Let $G$ be an admissible group. The following is a key concept introduced in \cite{repst}:

\begin{definition} \label{defn:meas}
A \defn{measure} for $G$ valued in a commutative ring $k$ is a rule $\mu$ assigning to each map $f \colon X \to Y$ in $\bS(G)$, with $Y$ transitive, a value $\mu(f)$ in $k$ such that the following axioms hold:
\begin{enumerate}
\item If $f$ is an isomorphism then $\mu(f)=1$.
\item Suppose $X=X_1 \amalg X_2$ is a decomposition of $G$-sets, and let $f_i$ be the restriction of $f$ to $X_i$. Then $\mu(f)=\mu(f_1)+\mu(f_2)$.
\item If $g \colon Y \to Z$ is a map of transitive $G$-sets then $\mu(g \circ f)=\mu(g) \cdot \mu(f)$.
\item Let $Y' \to Y$ be a map of transitive $G$-sets, and let $f' \colon X' \to Y'$ be the base change of $f$. Then $\mu(f')=\mu(f)$. \qedhere
\end{enumerate}
\end{definition}

\begin{remark}
Definiton~\ref{defn:meas} is taken from \cite[\S 4.5]{repst}. It is equivalent to \cite[Definition~3.1]{repst} by \cite[Proposition~4.10]{repst}.
\end{remark}

Let $\mu$ be a measure. For a finitary $G$-set $X$, we put $\mu(X)=\mu(X \to \bone)$, where $\bone$ is the 1-point $G$-set. We note that $\mu$ is additive, in the sense that $\mu(X \amalg Y)=\mu(X)+\mu(Y)$. If $f \colon X \to Y$ is a map of transitive $G$-sets then by considering the composition $X \to Y \to \bone$, we find $\mu(f) \cdot \mu(Y)=\mu(X)$.

We say that $\mu$ is \defn{regular} if $\mu(f)$ is a unit of $k$ whenever $f \colon X \to Y$ is a map of transitive $G$-sets. Suppose this is the case. By the previous paragraph, we find $\mu(f)=\mu(X) \mu(Y)^{-1}$, and so $\mu$ is determined by its values on $G$-sets. One easily sees that $\mu$ satisfies the identity
\begin{equation} \label{eq:reg}
\mu(X \times_W Y) = \mu(X) \cdot \mu(Y) \cdot \mu(W)^{-1},
\end{equation}
if $X \to W$ and $Y \to W$ are maps of transitive $G$-sets. In fact, this identity essentially characterizes regular measures. Precisely, if $\mu$ is a rule that assigns to each transitive $G$-set a unit of $k$ such that
\begin{enumerate}[(i)]
\item $\mu$ is invariant under isomorphism
\item $\mu(\bone)=1$
\item $\mu$ satisfies \eqref{eq:reg}, where we extend $\mu$ additively on the left side
\end{enumerate}
then $\mu$ defines a regular measure for $G$. See \cite[\S 4.6]{repst} for details.

Understanding measures for $G$ can be difficult in general. The following object can help with this problem.

\begin{definition}
We define a commutative ring $\Theta(G)$ as follows. For each map $f \colon X \to Y$ in $\bS(G)$, with $Y$ transitive, there is a class $[f]$ in $\Theta(G)$. These classes satisfy the following relations, using notation as in Definition~\ref{defn:meas}:
\begin{enumerate}
\item If $f$ is an isomorphism then $[f]=1$.
\item We have $[f]=[f_1]+[f_2]$.
\item We have $[g \circ f] = [g] \cdot [f]$.
\item We have $[f']=[f]$
\end{enumerate}
More precisely, $\Theta(F)=R/I$, where $R$ is the polynomial ring in symbols $[f]$ and $I$ is the ideal generated by the above relations.
\end{definition}

There is a measure $\mu_{\rm univ}$ valued in $\Theta(G)$ defined by $\mu_{\rm univ}(f)=[f]$. This measure is universal in the sense that if $\mu$ is a measure valued in a ring $k$ then there is a unique ring homomorphism $\phi \colon \Theta(G) \to k$ such that $\mu = \phi \circ \mu_{\rm univ}$. We thus see that understanding all measures for $G$ amounts to computing the ring $\Theta(G)$.

\subsection{Push-forwards} \label{ss:push}

Fix an admissible group $G$ and a measure $\mu$ for $G$ valued in a commutative ring $k$. For a finitary $G$-set $X$, we let $\cF(X)$ be the space of $G$-invariant functions $X \to k$. Equivalently, $\cF(X)$ is the space of functions $G \backslash X \to k$, which shows that $\cF(X)$ is a free $k$-module of finite rank. For a $G$-stable subset $W \subset X$, we let $1_W \in \cF(X)$ be the indicator function of $W$. The functions $1_W$, with $W$ a $G$-orbit on $X$, are a $k$-basis for $\cF(X)$. We note that $\cF(X)$ is the $G$-invariant subspace of the Schwartz space $\cC(X)$ discussed in \cite[\S 3.2]{repst}.

Suppose $f \colon X \to Y$ is a morphism in $\bS(G)$. We define a $k$-linear map $f_* \colon \cF(X) \to \cF(Y)$ by $f_*(1_W) = c \cdot 1_{f(W)}$, where $W$ is a $G$-orbit on $X$, and $c=\mu(f \colon W \to f(W))$. One should think of $f_*$ as ``integration over the fibers of $f$,'' and this is literally true according to the approach in \cite{repst}. This push-forward operation has many of the properties one would expect; see \cite[\S 3.4]{repst} for details.

\subsection{Matrices}

Maintain the set-up from \S \ref{ss:push}. Let $Y$ and $X$ be finitary $G$-sets. A \defn{$(Y \times X)$-matrix} is simply an element of $\cF(Y \times X)$. More generally, one could allow functions in the larger space $\cC(Y \times X)$, as in \cite{repst}, but this is unnecessary for our purposes. We write $\Mat_{Y,X}$ for the space of such matrices.

Suppose $A \in \Mat_{Y,X}$ and $B \in \Mat_{Z,Y}$. We define the \defn{product matrix} $BA \in \Mat_{Z,X}$ to be $(p_{13})_*(p_{23}^*(B) \cdot p_{12}^*(A))$. Here $p_{ij}$ is the projection from $Z \times Y \times X$ onto the $i$ and $j$ factors. Explicitly, $p_{23}^*(B) \cdot p_{12}^*(A)$ is the function $(z,y,x) \mapsto B(z,y) \cdot A(y,x)$. Matrix multiplication is associative, and the identity matrix $1_X \in \Mat_{X,X}$ is the identity for matrix multiplication. In particular, we see that $\Mat_{X,X}$ is an associative $k$-algebra.

There is one other operation on matrices we mention. Suppose $A \in \Mat_{Y,X}$ and $A' \in \Mat_{Y',X'}$. We define the \defn{Kronecker product} of $A$ and $A'$ to be the $(Y \times Y') \times (X \times X')$ matrix given by $(y,y',x,x') \mapsto A(y,x) \cdot A'(y',x')$.

\subsection{Permutation modules}

Maintain the set-up from \S \ref{ss:push}. In \cite[\S 8]{repst}, we define a $k$-linear tensor category $\uPerm_k(G;\mu)$ of ``permutation modules.'' We briefly summarize its definition and salient properties:
\begin{itemize}
\item The objects are formal symbols $\Vec_X$ where $X$ is a finitary $G$-set.
\item The morphisms $\Vec_X \to \Vec_Y$ are $(Y \times X)$-matrices.
\item Composition of morphisms is given by matrix multiplication.
\item Direct sums are given on objects by $\Vec_X \oplus \Vec_Y=\Vec_{X \oplus Y}$, and on morphisms by the usual block diagonal matrices.
\item The tensor product $\uotimes$ is given on objects by $\Vec_X \uotimes \Vec_Y=\Vec_{X \times Y}$, and on morphisms by the Kronecker product.
\item Every object is rigid and self-dual.
\end{itemize}
The category $\uPerm_k(G;\mu)$ is almost never abelian. However, in many cases it does admit an abelian envelope; this is discussed in \cite[Part~III]{repst}.

\subsection{The relative case}

A \defn{stabilizer class} in $G$ is a collection $\sE$ of subgroups satisfying the following conditions: (a) $\sE$ contains $G$; (b) $\sE$ is closed under conjugation; (c) $\sE$ is closed under finite intersections; (d) $\sE$ is a neighborhood basis of the identity. We say that a $G$-set is \defn{$\sE$-smooth} if every stabilizer belongs to $\sE$, and we write $\bS(G, \sE)$ for the category of finitary $\sE$-smooth $G$-sets. See \cite[\S 2.6]{repst} for more details.

Fix a stabilizer class $\sE$. The definitions we have made above for $G$ generalize to $(G,\sE)$ by simply considering $\sE$-smooth sets throughout; we call this the ``relative case'' in \cite{repst}. In particular, a measure for $(G,\sE)$ assigns to each morphism $f \colon X \to Y$ in $\bS(G,\sE)$, with $Y$ transitive, a quantity $\mu(f) \in k$, satisfying analogous conditions to those in Definition~\ref{defn:meas}. There is again a universal measure valued in a ring $\Theta(G,\sE)$.

\section{The group $G$} \label{s:G}

\subsection{Definition}

Let $\fX$ be a Cantor space, that is, a topological space homeomorphic to the Cantor set. Let $G$ be the group of all self-homeo\-morphisms of $\fX$. For a subset $\fA$ of $\fX$, let $G\{\fA\}$ be the subgroup of $G$ consisting of elements $g$ such that $g\fA=\fA$. Let $\sE$ be all subgroups of $G$ of the form $G\{\fA_1\} \cap \cdots \cap G\{\fA_n\}$ where $\fA_1, \ldots, \fA_n$ are clopen subsets of $\fX$. We topologize $G$ by taking $\sE$ to be a neighborhood basis of the identity. We will see below that $G$ is an admissible topological group. This group is our primary object of study.

\begin{remark}
A possible choice for $\fX$ is $\bZ_p$, the space of $p$-adic integers. In this model, the clopen sets have the form $\sum_{a \in S} (a+p^n \bZ_p)$, where $S$ is a subset of $\bZ/p^n \bZ$.
\end{remark}

\subsection{The $X$ sets}

For a non-empty finite set $A$, let $X(A)$ be the set of all surjective continuous maps $\fX \to A$, where $A$ is given the discrete topology. The group $G$ naturally acts on $X(A)$. We also put $X(n)=X([n])$ for $n \ge 1$, where $[n]=\{1,\ldots,n\}$. We can identify $X(n)$ with the set of tuples $(\fA_1, \ldots, \fA_n)$ where the $\fA_i$ are non-empty disjoint clopen subsets of $\fX$ such that $\fX=\bigcup_{i=1}^n \fA_i$.

\begin{proposition}
The action of $G$ on $X(A)$ is transitive.
\end{proposition}

\begin{proof}
Consider two elements $\phi \colon \fX \to A$ and $\psi \colon \fX \to A$ of $X(A)$. Let $\fA_i=\phi^{-1}(i)$ and $\fB_i=\psi^{-1}(i)$. These are non-empty clopen subsets of $\fX$, and thus Cantor spaces; let $\sigma_i \colon \fA_i \to \fB_i$ be a homeomorphism. Let $\sigma$ be the self-homeomorphism of $\fX$ defined by $\sigma \vert_{\fA_i} = \sigma_i$. Then $\phi = \psi \circ \sigma$, which completes the proof.
\end{proof}

\tikzset{leaf/.style={circle,fill=black,draw,minimum size=1mm,inner sep=0pt}}
\tikzset{omit/.style={circle,gray,draw,minimum size=1mm,inner sep=0pt}}
\begin{figure}
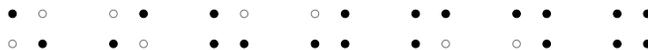

\begin{displaymath}
\tikz{
\node[omit] () at (0,0) {};
\node[leaf] () at (.4,0) {};
\node[leaf] () at (0,.4) {};
\node[omit] () at (.4,.4) {};
}
\qquad
\tikz{
\node[leaf] () at (0,0) {};
\node[omit] () at (.4,0) {};
\node[omit] () at (0,.4) {};
\node[leaf] () at (.4,.4) {};
}
\qquad
\tikz{
\node[leaf] () at (0,0) {};
\node[leaf] () at (.4,0) {};
\node[leaf] () at (0,.4) {};
\node[omit] () at (.4,.4) {};
}
\qquad
\tikz{
\node[leaf] () at (0,0) {};
\node[leaf] () at (.4,0) {};
\node[omit] () at (0,.4) {};
\node[leaf] () at (.4,.4) {};
}
\qquad
\tikz{
\node[leaf] () at (0,0) {};
\node[omit] () at (.4,0) {};
\node[leaf] () at (0,.4) {};
\node[leaf] () at (.4,.4) {};
}
\qquad
\tikz{
\node[omit] () at (0,0) {};
\node[leaf] () at (.4,0) {};
\node[leaf] () at (0,.4) {};
\node[leaf] () at (.4,.4) {};
}
\qquad
\tikz{
\node[leaf] () at (0,0) {};
\node[leaf] () at (.4,0) {};
\node[leaf] () at (0,.4) {};
\node[leaf] () at (.4,.4) {};
}
\end{displaymath}
\caption{The seven ample subsets of $[2] \times [2]$.}
\label{fig:ample}
\end{figure}

Given a surjection of finite sets $f \colon A \to B$, there is an induced map $X(f) \colon X(A) \to X(B)$. The map $X(f)$ is clearly $G$-equivariant, and thus surjective by the above proposition. If $A$ and $B$ are finite sets, we say that a subset of $A \times B$ is \defn{ample} if it surjects onto $A$ and $B$ under the projection maps. See Figure~\ref{fig:ample} for an example.

\begin{proposition} \label{prop:Xprod}
Let $A \to C$ and $B \to C$ be surjections of finite sets. Then the fiber product $X(A) \times_{X(C)} X(B)$ is naturally isomorphic to $\coprod_D X(D)$, where the coproduct is taken over ample subsets $D$ of $A \times_C B$.
\end{proposition}

\begin{proof}
Suppose that $D$ is an ample subset of $A \times_C B$. Applying the functor $X$ to the two projections $D \to A$ and $D \to B$, we obtain a map $X(D) \to X(A) \times_{X(C)} X(B)$, which is clearly injective; we thus regard $X(D)$ as a subset of $X(A) \times_{X(C)} X(B)$. One easily sees that it consists of those pairs $(\phi, \psi)$ such that the image of the map $\phi \times \psi \colon \fX \to A \times_C B$ is $D$; in particular, we see that $X(D)$ and $X(D')$ are disjoint if $D$ and $D'$ are distinct. If $(\phi, \psi)$ is an arbitrary element of $X(A) \times_{X(C)} X(B)$ then the image $D$ of $\phi \times \psi$ is an ample subset of $A \times_C B$, and so $(\phi, \psi)$ belongs to $X(D)$. The result follows.
\end{proof}

\begin{proposition} \label{prop:G-admissible}
The topological group $G$ is admissible.
\end{proposition}

\begin{proof}
The topology on $G$ is non-archimedean by definition. We now show that it is Hausdorff. For this, it suffices to show that $K=\bigcap G\{\fA\}$ is trivial, where the intersection is over all clopen subsets $\fA$. Suppose $g \in G$ is not the identity, and let $x \in \fX$ be such that $g(x) \ne x$. Pick a clopen subset $\fA$ of $\fX$ containing $x$ and not containing $gx$, which is possible since $\fX$ is Hausdorff and every point has a neighborhood basis of clopen sets. We thus see that $g\fA \ne \fA$, and so $g \not\in K$. Hence $K=1$, as required.

It remains to show that $G$ is Roelcke pre-compact. First, observe that if $\fA$ is a proper non-empty clopen subset of $\fX$ then $G/G\{\fA\}$ is isomorphic to $X(2)$ as a $G$-set. We thus see that if $U$ is any open subgroup in $\sE$ then $G/U$ is isomorphic to an orbit on $X(2)^n$ for some $n$. Now, let $U$ and $V$ be arbitrary open subgroups of $G$, let $U_0$ and $V_0$ be members of $\sE$ such that $U_0 \subset U$ and $V_0 \subset V$, and let $n$ and $m$ be such that $G/U_0 \subset X(2)^n$ and $G/V_0 \subset X(2)^m$. We have
\begin{displaymath}
U_0 \backslash G/V_0 = G \backslash (G/U_0 \times G/V_0) \subset G \backslash X(2)^{n+m}.
\end{displaymath}
Since $X(2)^{n+m}$ has finitely many orbits by Proposition~\ref{prop:Xprod}, we see that the above set is finite. Since $U \backslash G/V$ is a quotient of this set, it too is finite.
\end{proof}

\begin{remark} \label{rmk:growth}
The above proof shows that $G$ is oligomorphic via its action on $X(2)$. Let $f(n)$ be the number of $G$-orbits on $X(2)^n$; this is the number of ample subsets of $[2]^n$, i.e., subsets surjecting onto each of the $n$ factors. We have $f(1)=1$, $f(2)=7$, $f(3)=193$, and $f(4)=63775$. Since any subset of $[2]^n$ containing the points $(1, \ldots, 1)$ and $(2, \ldots, 2)$ is ample, we have $f(n) \ge \tfrac{1}{4} \cdot 2^{2^n}$. The sequence $f(n)$ is discussed at \cite{OEIS}.
\end{remark}

\begin{remark}
The group $G$ acts oligomorphically on the set $\fX$ as well. This action is highly transitive, that is, it is $n$-transitive for all $n$. It follows that $G$ is dense in the symmetric group on $\fX$. Thus, if we let $G'$ be the group $G$ equipped with the topology coming from its action on $\fX$, then we find that $\bS(G')$ is equivalent to $\bS(\fS)$, where $\fS$ is the infinite symmetric group. This implies that measures (and the resulting tensor categories) for $G'$ are the same as for $\fS$.
\end{remark}

\subsection{The $Y$ sets}

We now introduce a variant of the sets $X(A)$ that are sometimes more convenient. For a finite set $A$ define $Y(A)$ to be the set of all continuous maps $\fX \to A$. By considering the possible images of such maps, we find $Y(A) = \coprod_C X(C)$, where the coproduct is taken over non-empty subsets $C$ of $A$. We also put $Y(n)=Y([n])$, for $n \ge 0$. If $f \colon A \to B$ is any map of finite sets then there is an induced map $Y(f) \colon Y(A) \to Y(B)$; precisely, the restriction of $Y(f)$ to $X(C)$ is $X(f \vert_C) \colon X(C) \to X(f(C))$. One easily sees that $f \mapsto Y(f)$ is compatible with composition.

\begin{proposition} \label{prop:Yprod}
Let $A \to C$ and $B \to C$ be surjections of finite sets. Then the natural map  $Y(A \times_C B) \to Y(A) \times_{Y(C)} Y(B)$ is an isomorphism.
\end{proposition}

\begin{proof}
This is simply the mapping property for fiber products: giving a continuous map $\fX \to A \times_C B$ is equivalent to giving continuous maps $\fX \to A$ and $\fX \to B$ with equal compositions to $C$.
\end{proof}

\subsection{Classification of $\sE$-smooth $G$-sets}

It is clear that $\sE$ is a stabilizer class in $G$, and that $X(A)$ is an $\sE$-smooth $G$-set. The following proposition gives a complete description of the category $\bS(G,\sE)$ of $\sE$-smooth $G$-sets.

\begin{proposition} \label{prop:E-smooth}
We have the following:
\begin{enumerate}
\item Every transitive $\sE$-smooth $G$-set is isomorphic to some $X(n)$.
\item Every $G$-equivariant map $X(A) \to X(B)$ has the form $X(f)$ for a unique surjection $f \colon A \to B$.
\end{enumerate}
\end{proposition}

\begin{proof}
(a) Let $X$ be a transitive $\sE$-smooth $G$-set. Then $X$ is isomorphic to $G/U$ for some $U \in \sE$. By the reasoning from Proposition~\ref{prop:G-admissible}, we see that $X$ is an orbit on $X(2)^n$ for some $n$. By Proposition~\ref{prop:Xprod}, it thus follows that $X \cong X(n)$ for some $n$.

(b) Consider a $G$-equivariant map $\phi \colon X(A) \to X(B)$. Let $\Gamma \subset X(A) \times X(B)$ be the graph of $\phi$. Note that the projection $\Gamma \to X(A)$ is an isomorphism; in particular, $\Gamma$ is a transitive $G$-set. By Proposition~\ref{prop:Xprod}, we have $\Gamma=X(E)$ for an ample subset $E$ of $A \times B$. Since $\Gamma \times_{X(A)} \Gamma$ is transitive, Proposition~\ref{prop:Xprod} again implies that the projection $E \to A$ is an isomorphism. Inverting this projection and composing with the projection $E \to B$, we obtain a surjection $f \colon A \to B$. One easily sees that the graph of $X(f)$ is $\Gamma$, and so $\phi=X(f)$.
\end{proof}

\begin{remark}
The proposition shows that the full subcategory of $\bS(G,\sE)$ spanned by the transitive objects is equivalent to the category $\FS$ of finite sets and surjections. The category $\bS(G,\sE)$ can actually be directly constructed from $\FS$. This is discussed in \cite[\S 5]{bcat}; according to the terminology there, $\bS(G,\sE)$ is a B-category and $\FS^{\op}$ is the corresponding A-category.
\end{remark}

\subsection{Classification of $G$-sets}

We have seen that every transitive $\sE$-smooth $G$-set is isomorphic to some $X(n)$. However, there are transitive $G$-sets that are not $\sE$-smooth. For example, if $\Gamma$ is a subgroup of the symmetric group $\fS_n$, then the quotient $X(n)/\Gamma$ is a transitive $G$-set, and one can show that it fails to be $\sE$-smooth when $\Gamma$ is non-trivial. In fact, these examples are exhaustive:

\begin{proposition} \label{prop:Gset}
Every transitive $G$-set is isomorphic to some $X(n)/\Gamma$, with $\Gamma \subset \fS_n$.
\end{proposition}

One can also show that any map $X(n)/\Gamma \to X(m)/\Gamma'$ lifts to a map $X(n) \to X(m)$. Using this, one can show that $X(n)/\Gamma$ and $X(m)/\Gamma'$ are isomorphic if and only if $n=m$ and $\Gamma$ is conjugate to $\Gamma'$ inside of $\fS_n$. This completely classifies transitive $G$-sets. We will not need these additional results though, so we do not prove them.

Proposition~\ref{prop:Gset} can also be viewed as a classification of open subgroups of $G$. In particular, we find:

\begin{corollary} \label{cor:Gset}
Every open subgroup of $G$ contains a member of $\sE$ with finite index.
\end{corollary}

We note that Corollary~\ref{cor:Gset} follows from a result of Truss \cite[Corollary~3.8]{Truss}; in fact, Truss proves the stronger result that $G$ satisfies the strong small index property. Our arguments are more algebraic, and serve as a model for more general results in \cite[\S 4.3]{regcat}.

The proof will take a bit of work. A transitive $G$-set $X$ is isomorphic to $G/U$ for some open subgroup $U$. By definition, $U$ contains some subgroup $V \in \sE$. The transitive $G$-set $G/V$ is $\sE$-smooth, and thus isomorphic to $X(A)$ for some finite set $A$. We therefore have a surjection $X(A) \to X$, and so $X$ is isomorphic to $X(A)/R$, where $R=X(A) \times_X X(A)$ is the equivalence relation on $X(A)$ induced by the quotient map. To prove the proposition, we will class the equivalence relations on $X(A)$.

Let $R \subset X(A) \times X(A)$ be a $G$-stable subset. By Proposition~\ref{prop:Xprod}, there is a family $\cR=\cR(R)$ of ample subsets of $A \times A$ such that $R = \coprod_{E \in \cR} X(E)$. The following lemma characterizes equivalence relations from this point of view.

\begin{lemma} \label{lem:Gset-1}
The set $R$ is an equivalence relation if and only if $\cR$ satisfies the following conditions:
\begin{enumerate}
\item The diagonal subset of $A \times A$ belongs to $\cR$.
\item If $E \in \cR$ then $E^t \in \cR$, where $E^t=\{(y,x) \mid (x,y) \in E\}$ is the transpose of $E$.
\item Let $F \subset A \times A \times A$. If $p_{12}(F)$ and $p_{23}(F)$ belong to $\cR$ then so does $p_{13}(F)$. Here $p_{ij}$ is the projection onto the $i$ and $j$ factors.
\end{enumerate}
\end{lemma}

\begin{proof}
Suppose $R$ is an equivalence relation. The reflexive law for $R$ implies (a), while the symmetry law implies (b). Let $F$ as in (c) be given. Since $F$ surjects onto each of the three factors, $X(F)$ is a $G$-orbit on $X(A) \times X(A) \times X(A)$. The hypotheses on $F$ ensure that $p_{12}(X(F))$ and $p_{23}(X(F))$ are contained in $R$. The transitive law for $R$ thus ensures that $p_{13}(X(F))$ is contained in $R$, which means that $p_{13}(F)$ belongs to $\cR$. The reverse direction follows from similar considerations.
\end{proof}

Suppose $\Gamma$ is a subgroup of the symmetric group $\fS_A$. Then $\Gamma$ acts on $X(A)$, and thus induces an equivalence relation $R_{\Gamma}$ on $X(A)$; that is, $R_{\Gamma}(x,y)$ is true if $x$ and $y$ are in the same orbit. Let $\cR_{\Gamma}$ correspond to $R_{\Gamma}$. One easily sees that $\cR_{\Gamma} = \{E_{\sigma}\}_{\sigma \in \Gamma}$, where $E_{\sigma} \subset A \times A$ is the graph of the map $\sigma \colon A \to A$.

Fix a $G$-stable equivalence relation $R$ on $X(A)$ for the next two lemmas, and let $\cR=\cR(R)$ be the corresponding collection of subsets of $A \times A$. We say that $R$ is \defn{small} if for each $E \in \cR$, the projection map $p_1 \colon A \times A \to A$ restricts to a bijection $p_1 \colon E \to A$; note by the transpose condition (b) above, this implies that $p_2$ also induces a bijection $p_2 \colon E \to A$. The equivalence relations $R_{\Gamma}$ from the previous paragraph are clearly small. In fact, this is all of them:

\begin{lemma} \label{lem:Gset-2}
If $R$ is a small then $R=R_{\Gamma}$ for some subgroup $\Gamma \subset \fS_A$.
\end{lemma}

\begin{proof}
Let $E \in \cR$. Since the two projection maps $E \to A$ are bijections, it follows that $E=E_{\sigma}$ is the graph of some permutation $\sigma$ of $A$. Let $\Gamma$ be the set of $\sigma$'s arising in this way; this is a subset of $\fS_A$. The conditions in Lemma~\ref{lem:Gset-1} imply that $\Gamma$ is a subgroup: (a) implies that $1 \in \Gamma$, (b) that $\Gamma$ is closed under inversion, and (c) that $\Gamma$ is closed under multiplication. (We elaborate on the final point. Say $\sigma, \tau \in \Gamma$. Let $F \subset A \times A \times A$ be the fiber product of $E_{\sigma}$ and $E_{\tau}$. Then $p_{12}(F)=E_{\sigma}$ and $p_{23}(F)=E_{\tau}$ belong to $\cR$, and so $p_{13}(F)$ does as well. One easily sees that $p_{13}(F)=E_{\sigma \tau}$, and so $\sigma \tau \in \Gamma$.) Thus $R=R_{\Gamma}$, as required.
\end{proof}

The following is the key lemma:

\begin{lemma} \label{lem:Gset-3}
Suppose $R$ is not small. Then there is a surjection $A \to B$, that is not an isomorphism, such that $\cR$ contains all ample subsets of $A \times_B A$.
\end{lemma}

\begin{proof}
Let $E$ be a member of $\cR$ such that $p_1 \colon E \to A$ is not a bijection. Let $(a,b)$ and $(a,c)$ be elements of $E$ with $b \ne c$, and let $B$ be the quotient of $A$ where $b$ and $c$ are identified. Let $D_1, \ldots, D_7$ be the seven ample subsets of $\{b,c\} \times \{b,c\}$ (see Figure~\ref{fig:ample}), enumerated in any order, and let $D_0 \subset A \times A$ be the set of elements $(x,x)$ with $x \ne b,c$. Then $D_0 \cup D_i$, for $1 \le i \le 7$, are the ample subsets of $A \times_B A$. We claim that $\cR$ contains these seven sets.

We first handle $D_7$, which is the 4-element set $\{b,c\} \times \{b,c\}$. Consider the set
\begin{displaymath}
F = \{(x,y,x) \mid (x,y) \in E^t\} \cup \{(x,a,y) \mid x,y \in \{b,c\}\}.
\end{displaymath}
We have $p_{12}(F)=E^t$ and $p_{23}(F)=E$, and so $p_{13}(F)=D_0 \cup D_7$ belongs to $\cR$, as required.

We now handle the general case. Put
\begin{displaymath}
D_i^+ = \{(x,y,z) \mid \text{$(x,z) \in D_i$ and $y \in \{b,c\}$} \}.
\end{displaymath}
Note that $p_{12}(D_i^+)=p_{23}(D_i^+)=D_7$, while $p_{13}(D_i^+)=D_i$. Put
\begin{displaymath}
F' = \{(x,x,x) \mid \text{$x \in A$ and $x \ne b,c$} \} \cup D_i^+
\end{displaymath}
Then
\begin{displaymath}
p_{12}(F')=p_{23}(F')=D_0 \cup D_7, \qquad p_{13}(F')=D_0 \cup D_i.
\end{displaymath}
Since $D_0 \cup D_7$ belongs to $\cR$ by the previous paragraph, we find $D_0 \cup D_i$ belongs to $\cR$, as required.
\end{proof}

\begin{proof}[Proof of Proposition~\ref{prop:Gset}]
For a finite set $A$, consider the following statement:
\begin{itemize}
\item[$S(A)$:] Any $G$-equivariant quotient of $X(A)$ is isomorphic to one of the form $X(B)/\Gamma$, for some finite set $B$ and subgroup $\Gamma$ of $\fS_B$.
\end{itemize}
Since every transitive $G$-set is a quotient of some $X(A)$, it is enough to prove $S(A)$ for all $A$. We proceed by induction on the cardinality of $A$.

Let a finite set $A$ be given, and suppose $S(B)$ holds for $\# B < \# A$. Consider a quotient $X(A)/R$, where $R$ is a $G$-invariant equivalence relation on $A$. If $R$ is small then $X(A)/R=X(A)/\Gamma$ for some subgroup $\Gamma$ of $\fS_A$ by Lemma~\ref{lem:Gset-2}. Suppose $R$ is not small. Then, by Lemma~\ref{lem:Gset-3}, there is a proper surjection $A \to B$ such that $\cR=\cR(R)$ contains all ample subsets of $A \times_B A$. Let $R'=X(A) \times_{X(B)} X(A)$, which is also an equivalence relation on $X(A)$. From Proposition~\ref{prop:Xprod}, we see that $\cR'=\cR(R')$ is the collection of all ample subsets of $A \times_B A$. Thus $\cR' \subset \cR$, and so $R' \subset R$, and so $X(A)/R$ is a quotient of $X(A)/R' \cong X(B)$. Thus by $S(B)$, we see that $X(A)/R$ is of the form $X(C)/\Gamma$ for some $\Gamma \subset \fS_C$. We have thus proved $S(A)$, which completes the proof.
\end{proof}

\section{Measures for $G$} \label{s:meas}

\subsection{Elementary maps}

A surjection $f \colon A \to B$ of non-empty finite sets is \defn{elementary} if $\# A=1+\# B$. In this case, there is a unique point $x \in B$ such that $f^{-1}(x)$ has cardinality two; if $y \in B$ is not equal to $x$ then $f^{-1}(y)$ has cardinality one. We call $x$ the \defn{distinguished point} of $f$. Every surjection of finite sets can be expressed as a composition of elementary surjections.

Similarly, we say that a map $X(A) \to X(B)$ of $G$-sets is \defn{elementary} if $\# A=1+\# B$. Such a map has the form $X(f)$ where $f \colon A \to B$ is an elementary surjection. Every map of transitive $\sE$-smooth $G$-sets can be expressed as a composition of elementary maps. (Here we have appealed to Proposition~\ref{prop:E-smooth}.)

The following proposition, which describes the fiber product of elementary maps, will play an important role in our analysis of measures.

\begin{proposition} \label{prop:elem-fiber}
Consider a cartesian square
\begin{displaymath}
\xymatrix{
Z \ar[r] \ar[d] & X(A) \ar[d]^{\phi} \\
X(B) \ar[r]^-{\psi} & X(C) }
\end{displaymath}
where $\phi$ and $\psi$ are elementary maps. Write $\phi=X(f)$ and $\psi=X(g)$, where $f \colon A \to C$ and $g \colon B \to C$ are elementary surjections, and let $n$ be the common value of $\# A$, $\# B$, and $1+\# C$.
\begin{enumerate}
\item If $f$ and $g$ have different distinguished point then $Z \cong X(n+1)$.
\item If $f$ and $g$ have the same distinguished point then $Z \cong X(n+2) \amalg X(n+1)^{\amalg 4} \amalg X(n)^{\amalg 2}$.
\end{enumerate}
\end{proposition}

\begin{proof}
(a) The only ample subset of $A \times_C B$ is itself, and this set has cardinality $n+1$. The result thus follows by Proposition~\ref{prop:Xprod}.

(b) Let $x \in C$ be the common distinguished point of $f$ and $g$, and let $C_0=C \setminus x$. Then $A \times_C B$ is naturally identified with $C_0 \amalg (f^{-1}(x) \times g^{-1}(x))$. The ample subsets have the form $C_0 \amalg D$, where $D$ is an ample subset of $f^{-1}(x) \times g^{-1}(x)$. There are seven such $D$, as depicted in Figure~\ref{fig:ample}. The result again follows by Proposition~\ref{prop:Xprod}.
\end{proof}

\subsection{Construction of measures}

We now construct two important measures for $G$.

\begin{theorem} \label{thm:meas}
There exist unique regular $\bQ$-valued measures $\mu$ and $\nu$ for $G$ satisfying
\begin{displaymath}
\mu(X(n))=(-2)^{n-1}, \qquad \nu(X(n))=(-1)^{n-1}.
\end{displaymath}
\end{theorem}

\begin{proof}
Fix a non-zero rational number $\alpha$. For a transitive $\sE$-smooth $G$-set $X$, define $\rho_{\alpha}(X)=\alpha^{n-1}$, where $n$ is such that $X \cong X(n)$; such $n$ exists by Proposition~\ref{prop:E-smooth}. We extend $\rho_{\alpha}$ additively to all finitary $\sE$-smooth $G$-sets. Clearly, $\rho_{\alpha}$ is an isomorphism invariant and satifies $\rho_{\alpha}(\bone)=1$, since $\bone=X(1)$. Thus, by the discussion in \S \ref{ss:meas}, $\rho_{\alpha}$ defines a regular measure for $(G,\sE)$ if and only if
\begin{equation} \label{eq:rho}
\rho_{\alpha}(X \times_W Y) = \rho_{\alpha}(X) \cdot \rho_{\alpha}(Y) \cdot \rho_{\alpha}(W)^{-1}
\end{equation}
holds for all maps $X \to W$ and $Y \to W$ of transitive $\sE$-smooth $G$-sets. Since any map of transitive $\sE$-smooth $G$-sets is a composition of elementary maps, it suffices to treat the case where $X \to W$ and $Y \to W$ are elementary.

Suppose we are in this case. Then, up to isomorphim, we have $X=Y=X(n+1)$ and $W=X(n)$. The fiber product $Z=X \times_W Y$ is described in Proposition~\ref{prop:elem-fiber}. In case (a), \eqref{eq:rho} amounts to
\begin{displaymath}
\alpha^{n+1} = \alpha^n \cdot \alpha^n \cdot \alpha^{1-n},
\end{displaymath}
which automatically holds. In case (b), we obtain
\begin{displaymath}
\alpha^{n+2}+4\alpha^{n+1}+2\alpha^n = \alpha^n \cdot \alpha^n \cdot \alpha^{1-n},
\end{displaymath}
which simplifies to
\begin{displaymath}
(\alpha+2)(\alpha+1)=0
\end{displaymath}
We thus see that $\rho_{\alpha}$ defines a measure for $(G,\sE)$ if and only if $\alpha=-1$ or $\alpha=-2$.

We have thus shown that the formulas for $\mu$ and $\nu$ do indeed define regular measures for $(G,\sE)$. By \cite[\S 4.2(d)]{repst}, we have $\Theta(G) \otimes \bQ \cong \Theta(G,\sE) \otimes \bQ$; note that $\sE$ is a large stabilizer class by Corollary~\ref{cor:Gset}. Thus every $\bQ$-valued measure for $(G,\sE)$ extends uniquely to $G$; in particular, $\mu$ and $\nu$ admit unique extensions to $G$. By \cite[Corollary~4.5]{repst}, we have
\begin{displaymath}
\mu(X(A)/\Gamma) = \mu(X(A))/\# \Gamma,
\end{displaymath}
and similarly for $\nu$. This shows that $\mu$ and $\nu$ are regular measures for $G$.
\end{proof}

We make two comments regarding these measures. First, if $\Gamma$ is a subgroup of $\fS_n$, then
\begin{displaymath}
\mu(X(n)/\Gamma) = [\fS_n \colon \Gamma] \cdot \frac{(-2)^n}{n!}.
\end{displaymath}
This belongs to the ring $\bZ_{(2)}$ of rational numbers with odd denominator, and so we see that $\mu$ is actually valued in $\bZ_{(2)}$. Second, we have
\begin{displaymath}
\mu(X(n) \to X(m)) = (-2)^{n-m}, \qquad \nu(X(n) \to X(m)) = (-1)^{n-m}.
\end{displaymath}
This shows that $\mu$ and $\nu$ are $\bZ$-valued measures for $(G,\sE)$.

We now compute the measures of $Y(n)$ under $\mu$ and $\nu$. These computations (and some generalizations) will be important in \S \ref{s:mu} and \S \ref{s:nu}.

\begin{proposition} \label{prop:Ymeas}
For $n \ge 0$, we have
\begin{displaymath}
\mu(Y(n)) = \begin{cases}
1 & \text{if $n$ is odd} \\
0 & \text{if $n$ is even} \end{cases}
\qquad
\nu(Y(n)) = \begin{cases}
1 & \text{if $n \ge 1$} \\
0 & \text{if $n=0$}
\end{cases}
\end{displaymath}
\end{proposition}

\begin{proof}
From the definition of $Y(n)$, we have
\begin{displaymath}
[Y(n)] = \sum_{i=1}^n \binom{n}{i} [X(i)]
\end{displaymath}
in $\Theta(G)$. We thus find
\begin{displaymath}
\mu(Y(n)) = \sum_{i=1}^n \binom{n}{i} (-2)^{i-1} = -\frac{1}{2} \big( (1+(-2))^n - 1 \big),
\end{displaymath}
which agrees with the stated formula. If $n=0$ then clearly $[Y(n)]=0$ and so $\nu(Y(n))=0$ as well. If $n \ge 1$ then
\begin{displaymath}
\nu(Y(n)) = \sum_{i=1}^n \binom{n}{i} (-1)^{i-1} = - \big( (1+(-1))^n - 1 \big) = 1,
\end{displaymath}
which completes the proof.
\end{proof}

\subsection{Classification of measures}

We now determine the ring $\Theta(G)$.

\begin{theorem} \label{thm:theta}
The measures $\mu$ and $\nu$ induce a ring isomorphism
\begin{displaymath}
\Theta(G) \cong \bZ_{(2)} \times \bQ.
\end{displaymath}
\end{theorem}

\begin{proof}
Let $\pi_n \colon [n+1] \to [n]$ be the elementary surjection defined by $\pi_n(i)=i$ for $1 \le i \le n-1$ and $\pi_n(n)=\pi_n(n+1)=n$. Let $x_n$ be the class of the map $X(\pi_n) \colon X(n+1) \to X(n)$ in $\Theta(G,\sE)$. Since every surjection is a composition of elementary surjections, and every elementary surjection is isomorphic to some $\pi_n$, see that $\Theta(G,\sE)$ is generated by the classes $x_n$, for $n \ge 1$.

Let $n \ge 2$. We can find an elementary surjection $g \colon [n+1] \to [n]$ with distinguished point different from $n$. By Proposition~\ref{prop:elem-fiber}(a), we see that the base change of $X(\pi_n)$ along the map $X(g)$ is isomorphic to $X(\pi_{n+1})$. We conclude that $x_n=x_{n+1}$ in $\Theta(G,\sE)$. In particular, $\Theta(G,\sE)$ is generated by $x_1$ and $x_2$.

Now let $n \ge 1$, and consider the base change of $X(\pi_n)$ along $X(\pi_n)$. By Proposition~\ref{prop:elem-fiber}(b), we have
\begin{displaymath}
x_n=2+4x_{n+1}+x_{n+2} x_{n+1}.
\end{displaymath}
Here we have used the fact that for any map $\phi \colon X(n+2) \to X(n)$ we have $[\phi]=[\pi_{n+1}] [\pi_n]$ in $\Theta(G,\sE)$. Taking $n=1,2$ and appealing to the previous paragraph, we find
\begin{displaymath}
x_1=x_2=2+4x_2+x_2^2.
\end{displaymath}
We thus see that we have a surjective ring homomorphism
\begin{displaymath}
\bZ \times \bZ \cong \bZ[x]/((x+1)(x+2)) \to \Theta(G,\sE), \qquad x \mapsto x_1
\end{displaymath}
This map is an isomorphism: indeed, the maps $\mu,\nu \colon \Theta(G, \sE) \to \bZ$ provide an inverse.

Now, $\Theta(G)$ has no $\bZ$-torsion by \cite[Theorem~5.1]{repst} and, as we have already seen in the proof of Theorem~\ref{thm:meas}, $\Theta(G) \otimes \bQ = \Theta(G,\sE) \otimes \bQ$. We thus see that the map
\begin{displaymath}
\phi \colon \Theta(G) \to \bQ \times \bQ
\end{displaymath}
given by $(\mu,\nu)$ is an isomorphism onto its image. It thus suffices to show that the image is $\bZ_{(2)} \times \bQ$. We know that the image is contained in this subring, since $\mu$ is valued in $\bZ_{(2)}$. Put
\begin{displaymath}
c_n = \phi([X(n)/\fS_n]) = \tfrac{1}{n!} \phi([X(n)]) = \tfrac{1}{n!} \big( (-2)^n, (-1)^n \big).
\end{displaymath}
We have
\begin{displaymath}
c_n+(n+1)c_{n+1} = -\tfrac{1}{n!} \big( (-2)^n, 0 \big), \qquad
2c_n+(n+1)c_{n+1} = \tfrac{1}{n!} \big( 0, (-1)^n \big).
\end{displaymath}
Since these elements belong to $\im(\phi)$ and generate $\bZ_{(2)} \times \bQ$ as a ring, the result follows.
\end{proof}

The theorem shows that $\mu$ and $\nu$ are, in essence, the only two measures for $G$. More precisely, if $k$ is any field of characteristic~0 then there are exactly two measures for $G$ valued in $k$, namely $\mu$ and $\nu$.

\section{Representation theory at \texorpdfstring{$\mu$}{μ}} \label{s:mu}

\subsection{Preliminaries} \label{ss:mu-prelim}

Let $\cV$ be the following tensor category:
\begin{itemize}
\item The objects are symbols $V(A)$, where $A$ is a finite set.
\item A map $V(A) \to V(B)$ is a $B \times A$ matrix with entries in $\bF_2$.
\item Composition is matrix multiplication.
\item The tensor product is given by $V(A) \otimes V(B)=V(A \times B)$ on objects, and by the Kronecker product on morphisms.
\end{itemize}
Of course, $\cV$ is equivalent to the category of finite dimensional $\bF_2$-vector spaces (as a tensor category); $V(A)$ is essentially just a vector space equipped with a basis indexed by $A$.

Fix a field $k$ of characteristic~0. We let $k[\cV]$ be the $k$-linearization of $\cV$, as defined in \cite[\S 2.2]{dblexp}. We recall the basic features of this category:
\begin{itemize}
\item The objects of $k[\cV]$ are formal symbols $[V(A)]$.
\item We have $\Hom_{k[\cV]}([V(A)], [V(B)]) = k[\rM_{B,A}(\bF_2)]$. Here $\rM_{B,A}(\bF_2)$ is the set of $B \times A$ matrices, and $k[S]$ denotes the $k$-vector space with basis $S$.
\item The category $k[\cV]$ admits a tensor product $\uotimes$. On objects, it is given by the formula $[V(A)] \uotimes [V(B)]=[V(A \times B)]$.
\end{itemize}
We let $k[\cV]^{\sharp}$ be the additive-Karoubi envelope of $k[\cV]$; see \cite[\S 2.3]{dblexp} for details. Every object $X$ of $k[\cV]^{\sharp}$ admits the endomorphism $[0]$. This endomorphism is idempotent, and thus induces a decomposition
\begin{displaymath}
X = X_0 \oplus X_1
\end{displaymath}
where $X_0$ is the image of $[0]$, and $X_1$ is the image of $1-[0]$. We thus obtain a decomposition
\begin{displaymath}
k[\cV]^{\sharp} = k[\cV]^{\sharp}_0 \oplus k[\cV]^{\sharp}_1,
\end{displaymath}
The category $k[\cV]^{\sharp}_1$ is the same as the category called $k[\cV]^{\sharp}_{\chi}$ in \cite[\S 3.2]{dblexp}, where $\chi \colon \bF_2^{\times} \to k^{\times}$ is the unique (trivial) character. As explained in \cite[\S 3.2]{dblexp}, the two summands above are closed under the tensor product. By \cite[Theorem~3.12]{dblexp}, $k[\cV]^{\sharp}_1$ is a semi-simple pre-Tannakian category.

\begin{remark}
In fact, the category $k[\cV]^{\sharp}$ is semi-simple. The object $[V(\emptyset)]$ is simple, and $k[\cV]_0^{\sharp}$ consists of objects that are $[V(\emptyset)]$-isotypic. Thus $k[\cV]_1^{\sharp}$ consists of objects that do not contain $[V(\emptyset)]$ as a simple constituent.
\end{remark}

\subsection{Main results}

The following is the main result of \S \ref{s:mu}:

\begin{theorem} \label{thm:mu}
We have an equivalence of tensor categories $\uPerm_k(G; \mu)^{\rm kar} \cong k[\cV]^{\sharp}_1$ under which $\Vec_{Y(A)}$ corresponds to $[V(A)]_1$.
\end{theorem}

The most important consequence of this theorem is the following corollary:

\begin{corollary}
The category $\uPerm_k(G; \mu)^{\rm kar}$ is a semi-simple pre-Tannakian category.
\end{corollary}

\begin{proof}
As we already mentioned, $k[\cV]^{\sharp}_1$ is semi-simple and pre-Tannakian.
\end{proof}

We make a few more remarks concerning the theorem.
\begin{enumerate}
\item The theorem shows that $\End(\Vec_{Y(n)})$ is the quotient of the monoid algebra $k[\rM_n(\bF_2)]$ by the $k$-span of the element $[0]$ (which is a 2-sided ideal). The structure of this algebra is described in \cite{Kovacs}.
\item The simple objects of $\uPerm_k(G; \mu)^{\rm kar}$ naturally correspond to simple $k[\GL_n(\bF_2)]$-modules, for $n \ge 1$. This follows from \cite[Theorem~1.1]{Kuhn}; see also \cite[Remark~3.8]{dblexp}.
\item In \cite[Part~III]{repst}, we introduce the completed group algebra $A(G)$ of $G$, and the category $\uRep_k(G; \mu)$ of smooth $A(G)$-modules. Using the arguments there, one can show that $\uRep_k(G; \mu)$ is equivalent to the ind-completion of $\uPerm_k(G; \mu)^{\rm kar}$. This provides a concrete realization of the latter category.
\item The category $k[\cV]^{\sharp}_1$ was studied in \cite{dblexp} due to its fast (double exponential) growth. From the perspective of $\uPerm_k(G; \mu)$, this growth rate comes from Remark~\ref{rmk:growth}.
\item Since $\mu$ is valued in $\bZ_{(2)}$, it is possible to consider the category $\uPerm_k(G; \mu)$ when $k$ is a field of characteristic~2. This could be an interesting object for further study.
\end{enumerate}

\subsection{Proof}

We require a few lemmas before proving the theorem. The following is the key computation:

\begin{lemma} \label{lem:mu-1}
Let $f \colon A \to B$ be a map of finite sets, and let $D$ be the subset of $B$ consisting of points $b$ such that $f^{-1}(b)$ has odd cardinality. Then $Y(f)_*(1_{Y(A)})=1_{Y(D)}$.
\end{lemma}

\begin{proof}
We have $1_{Y(A)}=\sum_{C \subset A} 1_{X(C)}$, and so
\begin{displaymath}
Y(f)_*(1_{Y(A)}) = \sum_{C \subset A, C \ne \emptyset} (-2)^{\# C - \# f(C)} \cdot 1_{X(f(C))}.
\end{displaymath}
The power of $-2$ here is the measure of the map $X(C) \to X(f(C))$ induced by $f$. To compute this sum, we fix $E \subset B$ and determine the coefficient $\gamma(E)$  of $1_{X(E)}$. From the above formula, we find
\begin{displaymath}
\gamma(E) = (-2)^{-\# E} \sum_{C \subset A, f(C)=E} (-2)^{\# C}.
\end{displaymath}
We think of $C$ as varying over subsets of $f^{-1}(E)$ that meet each fiber. In other words, choosing $C$ amounts to choosing a non-empty subset of $f^{-1}(x)$ for each $x \in E$. We thus find
\begin{displaymath}
\gamma(E) = (-2)^{-\# E} \prod_{x \in E} \big( \sum_{S \subset f^{-1}(x), S \ne \emptyset} (-2)^{\# S} \big).
\end{displaymath}
As we saw in the proof of Proposition~\ref{prop:Ymeas}, the sum over $S$ is equal to $-2$ if $\# f^{-1}(x)$ is odd, and~0 otherwise. We thus see that $\gamma(E)=1$ if $E \subset D$, and $\gamma(E)=0$ otherwise. Hence
\begin{displaymath}
Y(f)_*(1_{Y(A)}) = \sum_{E \subset D} 1_{X(E)} = 1_{Y(D)},
\end{displaymath}
which completes the proof.
\end{proof}

For $\alpha \in \rM_{B,A}(\bF_2)$, we define a $Y(B) \times Y(A)$ matrix $\alpha^*$ as follows. Let $D \subset B \times A$ be the set of pairs $(y,x)$ such that $\alpha_{y,x} \ne 0$. We then define $\alpha^*=1_{Y(D)}$. Here we have identified $Y(B) \times Y(A)$ with $Y(B \times A)$ via Proposition~\ref{prop:Yprod}, which contains $Y(D)$ as a subset. Note that if $\alpha$ is the zero matrix then $D$ is empty, and so $\Lambda(\alpha)$ is also zero.

\begin{lemma} \label{lem:mu-2}
The construction $(-)^*$ is compatible with matrix multiplication.
\end{lemma}

\begin{proof}
Let $\alpha \in \rM_{B,A}(\bF_2)$ and $\beta \in \rM_{C,B}(\bF_2)$ be given. We show that $(\beta \alpha)^*=\beta^* \alpha^*$. Let $D \subset B \times A$ be the indices where $\alpha$ is~1, and let $E \subset C \times B$ and $F \subset C \times A$ be similarly defined for $\beta$ and $\beta \alpha$. Consider the projection map $\pi \colon E \times_B D \to C \times A$. The fiber $\pi^{-1}(z,x)$ is in bijection with the set of $y \in B$ such that $\alpha_{y,x}=1$ and $\beta_{z,y}=1$. Thus in the expression $(\beta \alpha)_{z,x} = \sum_{y \in B} \beta_{z,y} \alpha_{y,x}$, we see that the number of summands that are~1 is the cardinality of the set $\pi^{-1}(x,z)$. We thus see that $F$ is exactly the set of pairs $(z,x)$ for which $\pi^{-1}(z,x)$ has odd cardinality.

Now, the product $\beta^* \alpha^*$ is, by definition, the push-forward of the function $p_{23}^*(1_{Y(E)}) \cdot p_{12}^*(1_{Y(F)})$ on $Y(C) \times Y(B) \times Y(A)$ along $p_{13}$. This function is the indicator function of $Y(E) \times_{Y(B)} Y(D)$. The map $p_{13} \colon Y(E) \times_{Y(B)} Y(D) \to Y(C) \times Y(A)$ is identified with the map $Y(\pi) \colon Y(E \times_B D) \to Y(C \times A)$ via Proposition~\ref{prop:Yprod}. We have $Y(\pi)_*(1_{Y(E \times_B D)})=1_{Y(F)}$ by Lemma~\ref{lem:mu-1}. We thus find $\beta^* \alpha^*=1_{Y(F)}=(\beta \alpha)^*$, which completes the proof.
\end{proof}

\begin{lemma} \label{lem:mu-3}
The construction $(-)^*$ is compatible with Kronecker products.
\end{lemma}

\begin{proof}
Let $\alpha \in \rM_{B,A}(\bF_2)$, let $\alpha' \in \rM_{B',A'}(\bF_2)$, and let $\beta=\alpha' \otimes \alpha$ be the Kronecker product, which belongs to $\rM_{B' \times B, A' \times A}(\bF_2)$. Let $D \subset B \times A$ be the indices at which $\alpha$ is~1, and let $D' \subset B' \times A'$ and $E \subset (B' \times B) \times (A' \times A)$ be defined similarly for $\alpha'$ and $\beta$. By definition, $\beta_{(y',y),(x',x)}=\alpha'_{y',x'} \cdot \alpha_{y,x}$, and so we see that $E$ is identified with $D' \times D$ under the canonical isomorphism between $(B' \times B) \times (A' \times A)$ and $(B' \times A') \times (B \times A)$.

A similar computation shows that the Kronecker product of $1_{Y(D)} \in \Mat_{Y(B),Y(A)}$ and $1_{Y(D')} \in \Mat_{Y(B'),Y(A')}$ is $1_{Y(D) \times Y(D')}$, where we again use a similar identification of fourfold products. Since $Y(D) \times Y(D')=Y(D \times D')$, the result follows.
\end{proof}

\begin{proof}[Proof of Theorem~\ref{thm:mu}]
We have a functor $\Phi \colon k[\cV] \to \uPerm_k(G; \mu)$ defined on objects by $\Phi([V(A)])=\Vec_{Y(A)}$ and on morphisms by $\Phi([\alpha])=\alpha^*$. It is easy to see that $\Phi$ preserves identity maps, and it is compatible with composition by Lemma~\ref{lem:mu-2}. The functor $\Phi$ is also naturally symmetric monoidal; the key points here are Proposition~\ref{prop:Yprod} and Lemma~\ref{lem:mu-3}.

The functor $\Phi$ naturally induces a functor $\Phi' \colon k[\cV]^{\sharp}_1 \to \uPerm_k(G; \mu)^{\rm kar}$. The functor $\Phi'$ contains $\Vec_{Y(n)}$ in its image, and is thus essentially surjective, since every object of the target is a summand of sums of these basic objects. To complete the proof, it suffices to show that $\Phi'$ is fully faithful. For this, it is enough to show that $\Phi'$ induces an isomorphism
\begin{displaymath}
\Hom_{k[\cV]}([V(A)]_1, [V(B)]_1) \to \Hom(\Vec_{Y(A)}, \Vec_{Y(B)}).
\end{displaymath}
The source above has a basis consisting of the elements $[\alpha]_1$, as $\alpha$ varies over the non-zero elements of $\rM_{B,A}(\bF_2)$. The target has a basis consisting of the matrices $1_{Y(D)}$, with $D \subset B \times A$ non-empty; indeed, note that there is an upper-triangular change of basis between these functions and the similarly indexed basis $1_{X(D)}$. Since $\Phi'$ bijectively maps one basis to the other, the result follows.
\end{proof}

\section{Representation theory at \texorpdfstring{$\nu$}{ν}} \label{s:nu}

\subsection{Preliminaries}

Let $\BB$ be the Boolean semiring. This ring has two elements,~0 and~1. Addition and multiplication are commutative and associative, and given by
\begin{align*}
0+0 &=0, & 0+1=1+0 &=1, & 1+1 &=1 \\
0 \cdot 0 &=0, & 0 \cdot 1 =1 \cdot 0 &=0, & 1 \cdot 1 &=1
\end{align*}
Let $\cW$ be the following tensor category:
\begin{itemize}
\item The objects are symbols $W(A)$ where $A$ is a finite set.
\item A map $W(A) \to W(B)$ is a $B \times A$ matrix with entries in $\BB$.
\item Composition is matrix multiplication.
\item The tensor product is defined just as for $\cV$.
\end{itemize}
The category $\cW$ is equivalent to the category of finite free $\BB$-modules.

Fix a field $k$ of characteristic~0. Let $k[\cW]$ be the $k$-linearization of $\cW$, and let $k[\cW]^{\sharp}$ be its additive-Karoubi envelope. These categories carry a tensor product $\uotimes$. As with $\cV$, we have a decomposition
\begin{displaymath}
k[\cW]^{\sharp} = k[\cW]^{\sharp}_0 \oplus k[\cW]^{\sharp}_1.
\end{displaymath}
It is still true that each summand is closed under $\uotimes$. However, \cite[Theorem~3.12]{dblexp} does not apply in this context, and is in fact false: $k[\cW]^{\sharp}_1$ is not semi-simple. Indeed, the monoid algebra $k[\rM_3(\BB)]$ is not semi-simple by \cite{Bremner}.

\begin{remark}
One can view a $B \times A$ matrix with entries in $\BB$ as a correspondence between the sets $A$ and $B$, and matrix multiplication agrees with composition of correspondences. Thus $\cW$ is the correspondence category; see \cite{BT3} for details.
\end{remark}

\subsection{Main results} \label{ss:nu-thm}

The following is the main result of \S \ref{s:nu}:

\begin{theorem} \label{thm:nu}
We have an equivalence of tensor categories $\uPerm_k(G; \nu)^{\rm kar} \cong k[\cW]^{\sharp}_1$ under which $\Vec_{Y(A)}$ corresponds to $[W(A)]_1$.
\end{theorem}

We now explain an important consequence of the theorem. If every nilpotent endomorphism in $\uPerm_k(G; \nu)$ had trace~0 then the argument in \cite[\S 13.2]{repst} would show that $\uPerm_k(G; \nu)^{\rm kar}$ is semi-simple. Since $k[\cW]^{\sharp}_1$ is known to be non-semi-simple (e.g., by \cite{Bremner}), it follows that there are nilpotent endomorphisms in $\uPerm_k(G; \nu)$ with non-zero trace. In fact, $\Vec_{Y(3)}$ has a nilpotent endomorphism with non-zero trace, which one can write down explicitly using \cite{Bremner}. As a consequence, it follows that $\uPerm_k(G; \nu)$ cannot be embedded into any pre-Tannakian category.

Suppose that $H$ is an oligomorphic group equipped with a regular measure $\lambda$ valued in $k$. We showed, in \cite[Corollary~7.19]{repst}, that a nilpotent endomorphism in $\uPerm_k(H; \lambda)$ has trace~0 provided $\lambda$ satisfies Property~(P). Roughly speaking, (P) means that it is possible to reduce $\lambda$ modulo infinitely many primes. At the time, we thought this hypothesis seemed rather artificial. However, the case of $(G, \nu)$ shows that there may be something significant to this condition; indeed, since $\nu$ is valued in $\bQ$ and no proper subring, it cannot be reduced modulo any prime. (It is perhaps suggestive that that $\mu$ measure is better behaved and can be reduced modulo the prime~2.)

\subsection{Proof}

The proof of Theorem~\ref{thm:nu} is quite similar to that of Theorem~\ref{thm:mu}, so we omit most of the details in what follows.

\begin{lemma} \label{lem:nu-1}
Let $f \colon A \to B$ be a map of non-empty finite sets with image $D$. Then $Y(f)_*(1_{Y(A)})=1_{Y(D)}$.
\end{lemma}

\begin{proof}
We follow the proof of Lemma~\ref{lem:mu-1}. Let $\gamma(E)$ be the coefficient of $1_{X(E)}$ in the push-forward. We have
\begin{align*}
\gamma(E) &= (-1)^{\# E} \sum_{C \subset A, f(C)=E} (-1)^{\# C} \\
&= (-1)^{\# E} \prod_{x \in E} \big( \sum_{S \subset f^{-1}(x), S \ne \emptyset} (-1)^{\# S} \big).
\end{align*}
As we saw in Proposition~\ref{prop:Ymeas}, the sum above is $-1$ if $S$ is non-empty, and~0 otherwise. We thus see that $\gamma(E)=1$ if $E \subset D$, and $\gamma(E)=0$ otherwise. The result follows.
\end{proof}

For $\alpha \in \rM_{B,A}(\BB)$, we define a $Y(B) \times Y(A)$ matrix $\alpha^*$ just as before: letting $D \subset B \times A$ be the set of pairs $(y,x)$ such that $\alpha_{y,x} \ne 0$, we put $\alpha^*=1_{Y(D)}$.

\begin{lemma} \label{lem:nu-2}
The construction $(-)^*$ is compatible with matrix multiplication.
\end{lemma}

\begin{proof}
The proof is very similar to that of Lemma~\ref{lem:mu-2}, but it is worthwhile to go through the details. Let $\alpha \in \rM_{B,A}(\BB)$ and $\beta \in \rM_{C,B}(\BB)$ be given. Let $D \subset B \times A$ be the indices where $\alpha$ is~1, and let $E \subset C \times B$ and $F \subset C \times A$ be similarly defined for $\beta$ and $\beta \alpha$. Consider the projection map $\pi \colon E \times_B D \to C \times A$. As before, in the expression $(\beta \alpha)_{z,x} = \sum_{y \in B} \beta_{z,y} \alpha_{y,x}$ we find that the number of summands that are~1 is the cardinality of the set $\pi^{-1}(x,z)$. We thus see that $F$ is exactly the set of pairs $(z,x)$ for which $\pi^{-1}(z,x)$ is non-empty, i.e., $F$ is the image of $\pi$. This is where the present computation diverges from the one in Lemma~\ref{lem:mu-2}, due to the different addition law in $\BB$.

The product $\beta^* \alpha^*$ is the push-forward of $p_{23}^*(1_{Y(E)}) \cdot p_{12}^*(1_{Y(F)})$ on $Y(C) \times Y(B) \times Y(A)$ along $p_{13}$. This function is the indicator function of $Y(E) \times_{Y(B)} Y(D)$. The projection map $p_{13} \colon Y(E) \times_{Y(B)} Y(D) \to Y(C) \times Y(A)$ is identified with the map $Y(\pi) \colon Y(E \times_B D) \to Y(C \times B)$ via Proposition~\ref{prop:Yprod}. We have $Y(\pi)_*(1_{Y(E \times_B D)})=1_{Y(F)}$ by Lemma~\ref{lem:nu-1}. We thus find $\beta^* \alpha^*=1_{Y(F)}=(\beta \alpha)^*$, which completes the proof.
\end{proof}

\begin{lemma} \label{lem:nu-3}
The construction $(-)^*$ is compatible with Kronecker products.
\end{lemma}

\begin{proof}
The proof is identical to that of Lemma~\ref{lem:mu-3}.
\end{proof}

\begin{proof}[Proof of Theorem~\ref{thm:mu}]
The proof is identical to that of Theorem~\ref{thm:mu}.
\end{proof}

\end{document}